\documentclass[12pt]{article}
\usepackage{amssymb,amscd,amsthm,amsmath,color}
\usepackage{fullpage}
\newcommand{\PP}{\mathbb{P}}
\newcommand{\gG}{\mathbb{G}}
\newcommand{\OO}{\mathcal{O}}
\newcommand{\II}{\mathcal{I}}

\newcommand{\FF}{\mathcal{F}}

\newcommand{\aA}{\mathcal{A}}

\newcommand{\UU}{\mathcal{U}}
\newcommand{\CC}{\mathbb{C}}

\newcommand{\Hom}{\textsf{Hom}}

\newcommand{\Sym}{\text{Sym}}

\newtheorem{theorem}{Theorem}[section]
\newtheorem{lemma}[theorem]{Lemma}
\newtheorem{proposition}[theorem]{Proposition}
\newtheorem{corollary}[theorem]{Corollary}

\newtheorem{conjecture}[theorem]{Conjecture}

\theoremstyle{remark}
\newtheorem*{remark}{Remark}

\author{Eric Riedl and David Yang}

\title{Rational curves on general type hypersurfaces}

\begin{document}
\maketitle

\abstract{We develop a technique that allows us to prove results about subvarieties of general type hypersurfaces.  As an application, we use a result of Clemens and Ran to prove that a very general hypersurface of degree $\frac{3n+1}{2} \leq d \leq 2n-3$ contains lines but no other rational curves.}

\section{Introduction}
It is often useful for understanding the geometry of a variety to understand the rational curves on it.  As a basic first question, we may ask: what are the dimensions of these spaces of rational curves? In particular, we would like to understand when there exist rational curves of a given cohomology class.  We work over $\CC$ throughout.

Hypersurfaces are a natural first case for these sorts of questions, and much work has been done investigating the dimensions of spaces of rational curves on hypersurfaces.  Given a hypersurface $X$, let $R_e(X)$ be the space of integral, geometrically rational curves of degree $e$ on $X$.  There is a natural conjecture, made in special cases by several different people, for the dimension of $R_e(X)$ for $X$ very general.  The dimension of $R_e(\PP^n)$ is $(e+1)(n+1)-4$, and it is naively $ed+1$ conditions for a degree $e$ rational curve to lie on a given hypersurface of degree $d$.  Thus, we would expect the dimension of $R_e(X)$ to be $(e+1)(n+1)-4 - (ed+1) = e(n-d+1)+n-4$.

\begin{conjecture}
\label{ReConj}
The dimension of $R_e(X)$ is $\max \{e(n-d+1)+n-4, -1 \}$ if $(n,d) \neq (3,4)$.
\end{conjecture}

In the Fano range, there is the work of Harris-Roth-Starr and Beheshti-Kumar \cite{HRS, beheshtiKumar}, as well as our previous paper \cite{RiedlYang}.  In the Calabi-Yau range, there is the work on the Clemens Conjecture \cite{Katz1986, Johnsen1995, Johnsen1997, Cotterill04rationalcurves, Cotterill07rationalcurves} and the work of many different people on K3 surfaces, see \cite{LiLiedtke} for more details.  It is known that this conjecture is false for K3 surfaces, as they all contain rational curves.

Here, we will consider the general type range, where $d > n+1.$ Claire Voisin has conjectured \cite{voisin2003} the following, which follows from Conjecture \ref{ReConj}:

\begin{conjecture}
(Voisin) On a very general hypersurface of general type, the degrees of rational curves on the hypersurface are bounded.
\end{conjecture}

There has been quite a bit of previous work towards this conjecture.  Clemens \cite{clemensCurveHyp} proved bounds on the degree and genus of curves lying in very general hypersurfaces.  Generalizing Clemens' results to higher-dimensional subvarieties, Ein \cite{Ein1988} proved that if $d \geq 2n-k$ and $1 \leq k \leq n-3$ then a $k$-dimensional subvariety of a very general hypersurface must have a resolutions of singularities with effective canonical bundle.  If the inequality is strict, he proves that the canonical bundle must in fact be big.  Voisin \cite{voisin1996, voisin1998} improves his bound by $1$, i.e., she proves the same result for $d \geq 2n-1-k$, with $1 \leq k \leq n-3$ as before.  In the special case of curves, this proves that a very general hypersurface of degree $d \geq 2n-2$ in $\PP^n$, $n\geq 4$, contains no rational curves.

Pacienza \cite{Pacienza2d-3} proved that for $n \geq 5$ and $d \geq 2n-3$, the dimension of the space of rational curves on a very general hypersurface in $\PP^n$ is the expected dimension.  Clemens and Ran build on Voisin's and Pacienza's work and prove that in a hypersurface of high enough degree, all subvarieties without effective canonical bundle are contained in the locus swept out by lines (see Theorem \ref{ClemensRan} for a precise statement).  Results along these lines are particularly interesting in light of the Lang Conjecture, which predicts that these hypersurfaces will not have dense spaces of rational curves.  To our knowledge, this work was the best result on the spaces of rational curves on very general general type hypersurfaces.  For instance, Voisin \cite{voisinOpenConjecture} reposes the question as recently as last year.

We improve on Pacienza's result.

\begin{theorem}
\label{mainCurvesThm}
Let $X \subset \PP^n$ be a general hypersurface of degree $d$, with $\frac{3n+1}{2} \leq d \leq 2n-3 $.  Then $X$ contains lines but no other rational curves.
\end{theorem}

We prove this via a downward induction.  The idea is to use a simple proposition about Grassmannians to transport results about not being swept out by certain subvarieties in high dimensional projective space to results about containing no such subvarieties in lower dimensional projective spaces.  The technique also allows us to recover a slightly weaker version of Ein's and Voisin's results on subvarieties of very general hypersurfaces, as well as to prove that given a family of varieties and a dimension $n$, there is a degree $d_0$ such that a very general degree $d \geq d_0$ hypersurface in $\PP^n$ contains no subvarieties from that family.

\section{Acknowledgements}
We would like to acknowledge many helpful conversations with Joe Harris, Matthew Woolf, Jason Starr, Roya Beheshti, and Lawrence Ein.  The first author was partially supported by NSF grant DGE1144152.

\section{Proofs}

The proof of Theorem \ref{mainCurvesThm} has three main ingredients.  The first is a result of Clemens and Ran.

\begin{theorem} (\cite{clemensAndRan})
Let $X$ be a very general hypersurface of degree $d$ in $\PP^n$ and let $k$ be an integer.  Let $d \geq n$, $\frac{d(d+1)}{2} \geq 3n - 1 - k$ and let $f: Y \to Y_0$ be a desingularization of an irreducible subvariety of dimension $k$.  Let $t = \max(0,-d+n+1+\lfloor \frac{n-k}{2} \rfloor)$.  Then either $h^0(\omega_Y(t)) \neq 0$ or $Y_0$ is contained in the subvariety of $X$ swept out by lines.
\end{theorem}

If we take the special case where $Y$ is a rational curve, then we can see that for $t=0$, $h^0(\omega_Y(t)) = 0$.  Thus, $Y_0$ must be contained in the subvariety swept out by lines, provided $d \geq \frac{3n+1}{2}$.

\begin{corollary}
\label{ClemensRan}
Let $X \subset \PP^n$ be a very general degree $d$ hypersurface.  Then if $\frac{d(d+1)}{2} \geq 3n - 2$ and $ d \geq \frac{3n+1}{2}$, then any rational curve in $X$ is contained in the union of the lines on $X$.  In fact, for $n \geq 3$ and $d \geq \frac{3n+1}{2}$ then any rational curve in $X$ is contained in the union of the lines of $X$.
\end{corollary}
\begin{proof} Note that if $n \geq 3$, then $d \geq \frac{3n+1}{2}$ and 
\[ \frac{d(d+1)}{2} \geq \frac{(3n+1)(3n+3)}{8} = \frac{1}{8} (9n^2+12n+3) \geq n^2 \geq 3n-2 \] 
since $n \geq 3$.
\end{proof}

The second ingredient is a new result about the Fano scheme of lines on $X$.

\begin{theorem}
\label{noCurvesFano}
If $X \subset \PP^n$ is very general and $n \leq \frac{(d+2)(d+1)}{6}$, then $F_1(X)$ contains no rational curves.
\end{theorem}

Note that if $d \geq \frac{3n+1}{2}$, then $\frac{(d+2)(d+1)}{6} \geq \frac{9n^2+24n+15}{24} \geq n$. Our final ingredient is a known result about reducible conics.

\begin{theorem}
\label{noConics}
If $X \subset \PP^n$ is general and $d \geq \frac{3n}{2}-1$, then $X$ contains no reducible conics.
\end{theorem}

Using these three results, the Main Theorem follows easily.

\begin{proof}(Theorem \ref{mainCurvesThm}) Let $Y \subset X$ be the subvariety of $X$ swept out by lines. Let $U \to F_1(X)$ be the universal line on $X$.  Note that $U$ is smooth, since $X$ is very general.  Since there are no reducible conics in $X$, the natural map $U \to X$ is bijective on closed points.  We claim any map from $\PP^1 \to Y$ lifts to a map from $\PP^1 \to U$.  To see this, note that the result is clear unless the image of the $\PP^1$ lies in the singular locus of $Y$.  By repeatedly restricting the map $U \to Y$, we can obtain a map $\PP^1 \to U$.  But by Theorem \ref{ClemensRan}, any rational curve in $X$ has to lie in $Y$ and hence lift to $U$.  If it is not a line, it will not be contracted by the map $U \to F_1(X)$, and so we will get rational curves in $F_1(X)$.  However, this contradicts Theorem \ref{noCurvesFano}.
\end{proof}

\begin{remark}
In fact, as Ein pointed out to us, for $d \geq \frac{3n}{2}-1$ and $X \subset \PP^n$ general, the natural map $U \to X$ is an embedding.  To see this, note that a general hypersurface in this degree range will contain only lines $\ell$ with balanced normal bundle $N_{\ell/X} = \OO(-1)^{d-n+1} \oplus \OO^{2n-d-3}$ (since hypersurfaces in this degree range contain no planar double lines).  Thus, for a line of this form, the natural map on tangent spaces $T_{\ell} F_1(X) \to (N_{\ell/X})_p$ is injective, which means that the natural map on tangent spaces $T_{p,\ell} U \to T_p(X)$ is also injective, which means that the map $U \to X$ is an embedding.
\end{remark}

Note that the proof of Theorem \ref{mainCurvesThm} shows that the bounds in Corollary \ref{ClemensRan} are sharp, since our proof shows that there will be no rational curves contained in the subvariety of $X$ swept out by lines for $d \geq \frac{3n}{2} - 1$. However, we can see that for $d = \frac{3n}{2}-1$, $X$ will contain conics.  

To see this, note that the dimension of $R_e(X)$ is at least $e(n-d+1)+n-4$, and so for $e=2$, $d = \frac{3n}{2} - 1$, $R_2(X)$ has dimension at least $0$ if it is nonempty.  To see that it is nonempty note that the space of hypersurfaces containing a given smooth conic is irreducible of the expected dimension, and so the incidence-correspondence of smooth conics in a hypersurface is irreducible of the expected dimension.  $R_e(X)$ is then non-empty if we can exhibit a single smooth conic $C$ in a single hypersurface $X$ with $H^0(N_{C/X}) = 0$, which can be done by explicit computation. 

Thus it remains to prove Theorems \ref{noCurvesFano} and \ref{noConics}.  Theorem \ref{noConics} is substantially easier, and is probably well-known.  However, we offer a different proof from a typical proof involving Hilbert functions, one that will set up some of the ideas for the proof of Theorem \ref{noCurvesFano}.  The idea of this proof is to realize a very general $(p,X)$ with $p \in X$ as a hyperplane section of a higher-dimensional hypersurface, and show that as we vary the hyperplane section, the locus of such pairs containing a reducible conic will be of high codimension.

\begin{proof} (Theorem \ref{noConics})
%
Note that a general hypersurface in $\PP^n$ will be a hyperplane section of a hypersurface in $\PP^{d+1}$ (for $n < d$) with at most finitely many lines through any point $p \in Y$ by \cite{HRS}, Theorem 2.1.  Let $X$ in $\PP^n$ be a hyperplane section of such a hypersurface $Y$ in $\PP^{d+1}$ (for $n < d$).  Consider the subvariety of $\PP^n \times \PP^N$ \[ C_{n,d} = \overline{\{(p,X) | \: X \subset \PP^n, \: p \in X \text{ with a reducible conic in $X$ with node at $p$} \} }. \]

We aim to show that $C_{n,d}$ is codimension at least $n$ in $S = \{ (p,X) | \: p \in X \subset \PP^n \}$ at $(X,p)$, which will suffice to prove the result.  Let $Z$ be the set of $n$-planes in $\PP^{d+1}$ passing through $p$.  Let $Z' \subset Z$ be the closure of the subset of $n$-planes that contain a reducible conic in $Y$ with node at $p$. By upper semicontinuity of fiber dimension, it suffices to show that $Z'$ has codimension at least $n$ in $Z$ at (the $n$-plane containing) $X.$

The dimension of $Z$ will be $\dim \gG(n-1,d) = n(d-n+1)$.  Since $p$ has a $0$-dimensional family of lines through it, it follows that there will be a $0$-dimensional family of $2$-planes through $p$ that contain a reducible conic with node at $p$. Thus, the dimension of $Z'$ will be the dimension of the space of $n$-planes through $p$ containing a $2$-plane, or $\dim \gG(n-3,d-2) = 6+ (n-2)(d-2-(n-3)) = 6 + (n-2)(d-n+1) .$  This shows that the codimension of $Z'$ in $Z$ is
\[ n(d-n+1) - (n-2)(d-n+1) = 2(d-n+1) .\]
This will be at least $n$ precisely when
\[ 2(d-n+1) \geq n \]
or
\[ d \geq \frac{3n}{2}-1 .\]
\end{proof}

The proof of Theorem \ref{noCurvesFano} relies on the following result about Grassmannians, which is of independent interest.  Roughly speaking, it says that if you have a family of $m$-planes in $\PP^n$, then the family of $(m+1)$-planes containing at least one of them has smaller codimension in the Grassmannian.

\begin{proposition}
\label{GrassProp}
Let $m \leq n$.  Let $B \subset \gG(m,n)$ be irreducible of codimension at least $\epsilon \geq 1$.  Let $C \subset \gG(m-1,n)$ be a nonempty subvariety satisfying the following condition: $\forall c \in C$, if $b \in \gG(m,n)$ has $c \subset b$, then $b \in B$.  Then it follows that the codimension of $C$ in $\gG(m-1,n)$ is at least $\epsilon + 1$.
\end{proposition}

To see that Proposition \ref{GrassProp} is sharp, note that the dimension of the space of lines through a given point has dimension $n-1$ and codimension $2(n-1) - (n-1) = n-1$, while the space of $2$-planes through a given point has dimension $2(n-2)$ and codimension $3(n-2) - 2(n-2) = n-2$.

Note that the statement of Proposition \ref{GrassProp} also holds for families of planes containing a fixed linear space, as can be seen by identifying the space of $m$-planes containing a fixed $k$-plane $\Lambda$ with the space of $(m-k-1)$-planes in a linear space complementary to $\Lambda$.

The proof makes use of the following elementary Lemma.

\begin{lemma}
\label{twoCondLemma}
Suppose $B \subset \gG(m,n)$ and $C \subset \gG(m-1,n)$ are nonempty subvarieties satisfying the following two conditions:
\begin{enumerate}
\item $\forall c \in C$, if $b \in \gG(m,n)$ has $c \subset b$, then $b \in B$.
\item $\forall b \in B$, if $c \in \gG(m-1,n)$ has $c \subset b$, then $c \in C$.
\end{enumerate}
Then $B = \gG(m,n)$ and $C = \gG(m-1,n)$.
\end{lemma}
\begin{proof} (Lemma \ref{twoCondLemma})
Let $\Lambda \in B$, and let $\Phi \in \gG(m,n)$.  We will show $\Phi \in B$.  Let $k = m - \dim (\Lambda \cap \Phi)$.  Then there exists a sequence of $m$ planes $\Lambda = \Lambda_0, \Lambda_1, \cdots, \Lambda_k = \Phi$ such that $\dim(\Lambda_i \cap \Lambda_{i+1}) = m-1$.  By Condition 2, if $\Lambda_i \in B$, then $\Lambda_i \cap \Lambda_{i+1} \in C$ and, hence, by Condition 1, $\Lambda_{i+1} \in B$.  Since $\Lambda_0 = \Lambda$ was in $B$ by assumption, we see that $\Phi \in B$.  This shows $B = \gG(m,n)$, and it follows that $C = \gG(m-1,n)$.
\end{proof}

\begin{proof} (Proposition \ref{GrassProp})
We show that $\dim C \leq  m(n-m+1) - (\epsilon + 1)$.  Consider the incidence correspondence $\II = \{(b,c) | \: c \in C, b \in \gG(m,n), c \subset b \}$.  Note that if $(b,c) \in \II$, then necessarily $b \in B$.  Let $\pi_B: \II \to B$, $\pi_C: \II \to C$ be the projection maps.  The generic fiber of $\pi_B$ has dimension at most $m-1$, since if all the fibers had dimension $m$, both Conditions 1 and 2 of Lemma \ref{twoCondLemma} would be satisfied, contradicting $\epsilon \geq 1$.  Thus, $\dim \II \leq \dim B + m-1 \leq (m+1)(n-m) + m-1 - \epsilon$.  By the Condition on $C$, we see that the fibers of $\pi_C$ have dimension $n-m$, so $\dim \II = \dim C + n-m$.  Putting the two relations together gives
\[ \dim C + n-m = \dim \II \leq (m+1)(n-m) + m-1 - \epsilon \]
or
\[ \dim C \leq (m+1)(n-m) - (n-m) - \epsilon = m(n-m) + m-1 - \epsilon = m(n-m+1) - (\epsilon + 1) .\]
The result follows.
\end{proof}

Using Proposition \ref{GrassProp}, we can prove Theorem \ref{noCurvesFano}.

\begin{proof} (Theorem \ref{noCurvesFano})
First we find pairs $(n,d)$ for which $F_1(X)$ is general type.  Let $\UU$ be the incidence correspondence $\{ (\ell, X) | \ell \subset X \} \subset \gG(1,n) \times \PP^N$.  Note that $\UU$ is smooth because it is a projective bundle over $\gG(1,n)$.  Thus by Sard's theorem, for $X$ general $F_1(X)$ is the smooth expected-dimensional vanishing locus of a section of the vector bundle $\Sym^d(S^*)$.  Recall that $c(S^*) = 1 + \sigma_1+\sigma_{1,1}$.  We use the splitting principle to work out $c_1(\Sym^d(S^*))$.  Suppose $c(S^*) = (1+\alpha)(1+\beta)$.  Then $c(\Sym^d(S^*)) = \sum_{k=0}^d (1+\alpha)^k(1+\beta)^{d-k}$.  Counting the coefficients of $\alpha$ in the products, we see that
\[ c_1(\Sym^d(S^*)) = \sum_{k=0}^d k \sigma_1 = \frac{d(d+1)}{2} \sigma_1 .\]
The canonical bundle of $\gG(1,n)$ is $-(n+1)\sigma_1$.  Thus, the canonical bundle of $F_1(X)$ is $(-n-1+\frac{d(d+1)}{2})\sigma_1$.  This will be ample if
\[ -n-1+\frac{d(d+1)}{2} \geq 1 \]
or
\[ n \leq \frac{d(d+1)}{2}-2 .\]
Let $m = \frac{d(d+1)}{2}-2$.  Consider 
\[ R_{n,d} = \{ (\ell, X) | \ell \subset X \subset \PP^n, \exists \text{ rational curve in $F_1(X)$ containing $[\ell]$} \}  \subset \gG(1,n) \times \PP^N .\]
Note that $R_{n,d}$ is a possibly countable union of irreducible varieties.  We claim that $R_{n,d}$ is codimension at least $2n-d-2$ in $\UU$ for $d \geq \frac{3n+1}{2}$, which will prove that a very general hypersurface $X \subset \PP^n$ has no rational curves in its Fano scheme.

Let $(\ell_0,X_0) \in R_{n,d}$.  We find a family $\FF \subset \UU$ with $(\ell_0,X_0) \in \FF \subset \{(\ell,X) | \ell \subset X \}$ such that $\FF \cap R_{n,d} \subset \FF$ is codimension at least $2n-d-2$ in $\FF$.  Let $Y' \subset \PP^m$ be a general hypersurface of degree $d$, and $\ell' \subset Y'$ a very general line in $Y$.  Note that there are no rational curves in $F_1(Y')$ through $[\ell']$ since $F_1(Y')$ is smooth and general type.  Let $Y \subset \PP^M$ (for some large $M$) be a hypersurface containing a line $\ell'_0$ such that $(\ell_0,X)$ is a $n$-plane section of $(\ell'_0, Y)$ and $(\ell',Y')$ is an $m$-plane section of $(\ell'_0,Y)$.  

Let $Z_r \subset \Hom(\PP^r, \PP^M)$ be the set of parameterized $r$-planes in $\PP^M$ containing $\ell'_0$, and let $Z'_r \subset Z_r$ be the set of parameterized $r$-planes containing $\ell$ such that the corresponding linear section of $Y$ has no rational curve through $\ell'_0$ in its Fano scheme.  We let $\FF$ be the image of $Z_n$ in the space of pairs $(\ell,X)$.  Then $\FF \cap R_{n,d}$ will be the image of $Z'_n$ in the space of hypersurfaces of degree $d$ in $\PP^n$.  It suffices to show that $Z'_n$ is codimension at least $2n-d-2$ in $Z_n$.  We see that $Z'_m$ is codimension at least one in $Z_m$.  By Proposition \ref{GrassProp}, we see that $Z'_n$ is codimension at least $m-n+1$ in $Z_n$.  Thus, our result will hold for 
\[ m-n+1 \geq 2n-d-2 \]
or
\[ 3n \leq m + d + 3 \]
or equivalently
\[ n \leq \frac{m+d}{3}+1 = \frac{(d+1)(d+2)}{6} .\]
\end{proof}

\begin{remark}
We should mention that as Ein pointed out to us, Theorem \ref{noCurvesFano} can also be proven using techniques similar to those in \cite{Ein1988} and \cite{Pacienza2d-3}.  We omit the details of a complete proof along these lines, but in broad outlines, the technique is to show that for $\aA$ the universal line on a hypersurface, the bundle $T_{\aA}|_{F_1(X)}(1)$ is globally generated, for $X$ very general.  This means that 
\[ \wedge^{2n-d-4} T_{\aA}|_{F_1(X)} (2n-d-4) \cong \Omega_{\aA}^{N+1}|_{F_1(X)} \left(-\frac{d(d+1)}{2}+n+1 + 2n-d-4 \right) \]
is globally generated.  Thus, $\Omega_{\aA}^{N+1}|_{F_1(X)}$ will be globally generated for 
\[ -\frac{d(d+1)}{2}+n+1 + 2n-d-4 \leq 0 \]
or
\[ 3n \leq \frac{d^2+d+2d+6}{2} = \frac{d^2+3d+6}{2} \]
or
\[ n \leq \frac{d^2+3d+6}{6} \]
which is approximately the bound we get by our technique.

As far as we are aware, this technique, while yielding a similar bound, is different in nature from our technique involving Grassmannians, even though we can recover Ein's results using the Grassmannian technique, and this global generation technique can recover these results about the Grassmannian.
\end{remark}

The techniques here can be used to prove that a general hypersurface in sufficiently high degree contains no copies of a class of subvarieties, provided that there exists a degree such that a general hypersurface is not swept out by these subvarieties.  When combined with a result of Voisin \cite{voisin2004}, this can be a useful way to prove that given any class of varieties and fixed dimension of hypersurface, there exists a degree large enough so that a very general hypersurface contains no varieties in that class.

\begin{theorem}
\label{generalTheorem}
Suppose that a very general complete intersection of degree $(d_1, \cdots, d_c)$ in $\PP^{2m-c-k}$ is not rationally swept out by varieties in the family $Y \to B$ with $Y_b$ $k$-dimensional.  Then a very general complete intersection of degree $(d_1, \cdots, d_c)$ in $\PP^m$ contains no varieties from the family $Y \to B$.
\end{theorem}
\begin{proof}
Consider the components of the incidence correspondence 
\[ \Gamma_n = \{ (p,X) | \: p \in X \subset \PP^n, p \text{ lies on a subvariety swept out by varieties of } Y \}. \]
We see that this incidence correspondence will have countably many irreducible components.  We show that $\Gamma_m$ will have codimension at least $m+1-c-k$ in the incidence correspondence
\[ \II_m = \{ (p,X) | \: p \in X \subset \PP^m \} .\]

We use the same technique as in the proof of Theorem \ref{noCurvesFano}.  Let $(p,X_0) \in \Gamma_m$ be general, let $(p,X_1) \in \Gamma_{2m-k}$ be general and let $(p,X_2) \in \Gamma_{3m-k}$ a pair such that $X_0$ and $X_1$ are $m$-plane sections of $X_2$ by a plane through $p$.  Let $S \subset \II_m$ be the space of pairs $(p,X)$ that are $m$-plane sections of $X_2$ by a plane through $p$.  Then by Proposition \ref{GrassProp} we see that $S \cap \Gamma_m \subset S$ is codimension at least $2m-c-k-m + 1 = m+1-c-k$.  This concludes the proof.
\end{proof}

This allows us to recover the results of Ein \cite{Ein1988}.

\begin{corollary}
A very general complete intersection of degree $(d_1, \cdots, d_c)$ in $\PP^n$ will contain no varieties that do not admit a resolution with effective canonical bundle if $\sum_{i=1}^c d_i \geq 2n-k-c+1$.  If the inequality is strict, then every subvariety will be general type.
\end{corollary}
\begin{proof}
Note that if a family of varieties sweeps out a Calabi-Yau complete intersection, then a general element of the family must have effective canonical bundle.  Taking $\sum_{i=1}^c d_i-1 = 2n-c-k$ from the corollary, we get our result if $\sum_{i=1}^c d_i \geq 2n-k-c+1$.

Using the fact that a family sweeping out a general type complete intersection must be general type and noting that complete intersections of degree $(d_1, \cdots, d_c)$ with $d = \sum d_i$ in $\PP^{d-2}$ are general type we obtain the second part of the statement.
\end{proof}

Using a theorem of Voisin \cite{voisin2004}, we obtain another immediate corollary.

\begin{theorem}
(Voisin \cite{voisin2004}) Let $n$ and $S \to B$ be given.  Then there exists $d$ sufficiently large so that a very general degree $d$ hypersurface in $\PP^n$ is not rationally swept by varieties in the family $S \to B$.
\end{theorem}

\begin{corollary}
\label{voisinCor}
Let $n$ and $S \to B$ be given.  Then there exists a $d$ sufficiently large so that a very general degree $d$ hypersurface in $\PP^n$ admits no finite maps from varieties in the fibers of $S \to B$.
\end{corollary}

\bibliographystyle{plain}
\bibliography{breakingandborrowing}
\end{document}